\documentclass[12pt]{article}
\usepackage{graphicx}
\usepackage{amsmath}
\usepackage{amsfonts}
\usepackage{amsthm}
\usepackage{amssymb}
\usepackage[T1]{fontenc}
\usepackage{url}
\usepackage{color}
\usepackage[margin=1in]{geometry}
\usepackage{tikz}
\numberwithin{equation}{section}
\usepackage{amssymb,bm}
\usepackage{amsmath}
\usepackage{amsthm}
\usepackage{graphicx}
\usepackage[active]{srcltx} 
\usepackage{hyperref}
\usepackage{ulem}
\hypersetup{pdfborder=0 0 0}
\usepackage{tikz}
\usepackage{tkz-fct} 
\usepackage{caption}
\usepackage{subcaption}

\newtheorem{theorem}{{\sc Theorem}}[section]

\newtheorem{lemma}[theorem]{{\sc Lemma}}

\newtheorem{remark}[theorem]{Remark}

\newtheorem{definition}[theorem]{Definition}

\newcommand{\supp}{\operatorname{supp}}

\def\XXint#1#2#3{{\setbox0=\hbox{$#1{#2#3}{\int}$ }
\vcenter{\hbox{$#2#3$ }}\kern-.6\wd0}}

\newcommand{\Gl}{\lambda}

\newcommand{\Go}{\omega}

\newcommand{\GO}{\Omega}

\bmdefine\BGa{\alpha}
\bmdefine\BGb{\beta}
\bmdefine\BGd{\delta}
\bmdefine\BGe{\epsilon}
\bmdefine\BGve{\varepsilon}
\bmdefine\BGf{\phi}
\bmdefine\BGvf{\varphi}
\bmdefine\BGg{\gamma}
\bmdefine\BGc{\chi}
\bmdefine\BGi{\iota}
\bmdefine\BGk{\kappa}
\bmdefine\BGl{\lambda}
\bmdefine\BGn{\eta}
\bmdefine\BGm{\mu}
\bmdefine\BGv{\nu}
\bmdefine\BGp{\pi}
\bmdefine\BGth{\theta}
\bmdefine\BGvth{\vartheta}
\bmdefine\BGr{\rho}
\bmdefine\BGvr{\varrho}
\bmdefine\BGs{\sigma}
\bmdefine\BGvs{\varsigma}
\bmdefine\BGt{\tau}
\bmdefine\BGj{\tau}
\bmdefine\BGu{\upsilon}
\bmdefine\BGo{\omega}
\bmdefine\BGx{\xi}
\bmdefine\BGy{\psi}
\bmdefine\BGz{\zeta}
\bmdefine\BGD{\Delta}
\bmdefine\BGF{\Phi}
\bmdefine\BGG{\Gamma}
\bmdefine\BGL{\Lambda}
\bmdefine\BGP{\Pi}
\bmdefine\BGT{\Theta}
\bmdefine\BGS{\Sigma}
\bmdefine\BGU{\Upsilon}
\bmdefine\BGO{\Omega}
\bmdefine\BGX{\Xi}
\bmdefine\BGY{\Psi}

\bmdefine\BCA{{\mathcal A}}
\bmdefine\BCB{{\mathcal B}}
\bmdefine\BCC{{\mathcal C}}
\bmdefine\BCD{{\mathcal D}}
\bmdefine\BCE{{\mathcal E}}
\bmdefine\BCF{{\mathcal F}}
\bmdefine\BCG{{\mathcal G}}
\bmdefine\BCH{{\mathcal H}}
\bmdefine\BCI{{\mathcal I}}
\bmdefine\BCJ{{\mathcal J}}
\bmdefine\BCK{{\mathcal K}}
\bmdefine\BCL{{\mathcal L}}
\bmdefine\BCM{{\mathcal M}}
\bmdefine\BCN{{\mathcal N}}
\bmdefine\BCO{{\mathcal O}}
\bmdefine\BCP{{\mathcal P}}
\bmdefine\BCQ{{\mathcal Q}}
\bmdefine\BCR{{\mathcal R}}
\bmdefine\BCS{{\mathcal S}}
\bmdefine\BCT{{\mathcal T}}
\bmdefine\BCU{{\mathcal U}}
\bmdefine\BCV{{\mathcal V}}
\bmdefine\BCW{{\mathcal W}}
\bmdefine\BCX{{\mathcal X}}
\bmdefine\BCY{{\mathcal Y}}
\bmdefine\BCZ{{\mathcal Z}}

\bmdefine\Bzr{ 0}
\bmdefine\Ba{ a}
\bmdefine\Bb{ b}
\bmdefine\Bc{ c}
\bmdefine\Bd{ d}
\bmdefine\Be{ e}
\bmdefine\Bf{ f}
\bmdefine\Bg{ g}
\bmdefine\Bh{ h}
\bmdefine\Bi{ i}
\bmdefine\Bj{ j}
\bmdefine\Bk{ k}
\bmdefine\Bl{ l}
\bmdefine\Bm{ m}
\bmdefine\Bn{ n}
\bmdefine\Bo{ o}
\bmdefine\Bp{ p}
\bmdefine\Bq{ q}
\bmdefine\Br{ r}
\bmdefine\Bs{ s}
\bmdefine\Bt{ t}
\bmdefine\Bu{ u}
\bmdefine\Bv{ v}
\bmdefine\Bw{ w}
\bmdefine\Bx{ x}
\bmdefine\By{ y}
\bmdefine\Bz{ z}
\bmdefine\BA{ A}
\bmdefine\BB{ B}
\bmdefine\BC{ C}
\bmdefine\BD{ D}
\bmdefine\BE{ E}
\bmdefine\BF{ F}
\bmdefine\BG{ G}
\bmdefine\BH{ H}
\bmdefine\BI{ I}
\bmdefine\BJ{ J}
\bmdefine\BK{ K}
\bmdefine\BL{ L}
\bmdefine\BM{ M}
\bmdefine\BN{ N}
\bmdefine\BO{ O}
\bmdefine\BP{ P}
\bmdefine\BQ{ Q}
\bmdefine\BR{ R}
\bmdefine\BS{ S}
\bmdefine\BT{ T}
\bmdefine\BU{ U}
\bmdefine\BV{ V}
\bmdefine\BW{ W}
\bmdefine\BX{ X}
\bmdefine\BY{ Y}
\bmdefine\BZ{ Z}

\begin{document}

\title{Korn and Poincar\'e-Korn inequalities for thin domains in $GSBD^p$}

\author{Davit Harutyunyan\thanks{University of California Santa Barbara, harutyunyan@math.ucsb.edu}}

\maketitle

\begin{abstract}

We prove Korn's first inequalities with and without boundary conditions in $GSBD^p$ ($1<p<\infty$) in $C^{1,1}$-regular thin domains with Lipschitz thickness boundaries with a constant $C h^{-p},$ that is optimal if the domain mid-surface contains a flat region, e.g., if the domain contains a piece of plate. For the one with boundary condition, we impose zero Dirichlet boundary conditions on the thin lateral faces of the shell, where we prove that the same estimate holds with no subtracted rigid motion like in the classical case for Sobolev vector fields. 

\end{abstract}

\section{Introduction}
\label{Sec:1}

The classical Korn's inequality Introduced by Korn [\ref{bib:Korn.1},\ref{bib:Korn.2}],  bounds the gradient of a displacement field $u$ by its symmetrized gradient $e(u) = \frac12(\nabla u + (\nabla u)^{T})$, after a constant skew-symmetric matrix is subtracted. It asserts the following: \textit{Let $1<p<\infty.$ For any bounded connected open Lipschitz domain $\GO \subset \mathbb{R}^n$, there is a constant $C(\GO)>0$ such that for any 
$u \in W^{1,p}(\GO,\mathbb{R}^n)$ one has}
\begin{equation}
\label{1.1}
\inf_{\BA \in \mathbb{R}^{n\times n}_{\mathrm{skew}}}
\int_\Omega |\nabla u - \BA|^p dx \le C(\GO)\int_\Omega |e(u)|^p dx,
\end{equation}
The inequality proves coercivity of the elastic energy in linear elasticity. When the domain is thin, e.g., a plate or a shell of thickness $h,$ then the optimal constant $K(\Omega)$ in (\ref{1.1}) blows up as $h\to 0,$ e.g., [\ref{bib:Fri.Jam.Mue.1},\ref{bib:Fri.Jam.Mor.Mue.},\ref{bib:Harutyunyan.1},\ref{bib:Harutyunyan.2}], and the sharp blow--up rate encodes the mechanical rigidity (as well as mathematical rigidity), i.e., the relative softness of bending against stretching [\ref{bib:Fri.Jam.Mue.1},\ref{bib:Fri.Jam.Mor.Mue.},\ref{bib:Harutyunyan.1},\ref{bib:Gra.Har.1},\ref{bib:Gra.Har.2},\ref{bib:Harutyunyan.2}]. While the asymptotics of the constant $C$ in $h$ in the nonlinear counterpart--the geometric rigidity estimate is central in dimension reduction theories of plates and shells [\ref{bib:Fri.Jam.Mue.1},\ref{bib:Fri.Jam.Mor.Mue.}], the asymptotics of the optimal constant in the Korn's first inequality (\ref{1.1}) is crucial in the theory of buckling of slender structures  [\ref{bib:Gra.Har.1},\ref{bib:Gra.Har.2},\ref{bib:Harutyunyan.2},\ref{bib:Har.Rod.1}]. For plates the optimal  constant in (\ref{1.1}) scales as $h^{-p}$ attained by Kirchhoff bending Ans\"atze, and for most regular enough shells, the existing theory suggest that the scaling is the same $h^{-p},$ see [\ref{bib:Harutyunyan.2}]. If one imposes zero Dirichlet boundary conditions on Sobolev vector fields on thin face of the shell boundary, then inequality (\ref{1.1}) holds true with the same scaling $h^{-p}$ and with $\BA=0,$ i.e., with no skew-symmetric matrix subtraction. This is due to the fact, that no infinitesimal rigid motion $\BA x+b$ can vanish on a positive $\mathcal{H}^2$-Hausdorff measure set.

The target of the present manuscript is to study Korn's first inequality (\ref{1.1}) for shells in the context of Griffiths model for fracture, i.e., for the so called generalized special functions of bounded deformation, that allow for jumps unlike Sobolev functions. In the companion paper [\ref{bib:Harutyunyan.3}], we use the obtained estimates to analyze the Griffith energy on shells in the context of deformation versus fracture, where we prove various $\Gamma-$convergence results and identify the shell toughness regimes where fracture happens before bending and vice versa. We briefly review these spaces and explain why they, rather than $W^{1,p}$, are the natural setting. In the linearized theory, a deformation with fractures cannot be modeled by a Sobolev field on the entire domain, as it jumps across the crack. The classical replacement is the space $BD(\GO)$ of functions of bounded deformation introduced by Matthies, Strang, and Christiansen  in [\ref{bib:Mat.Str.Chr.}] and by Suquet in [\ref{bib:Suquet.1}], see also [\ref{bib:Tem.Str.},\ref{bib:Suquet.2}]. The space consists of those functions $u \in L^1(\GO;\mathbb{R}^n)$ whose symmetrized distributional derivative $Eu = \frac12\big(Du + (Du)^{T}\big)$ is a bounded Radon measure. It is essential here that only the symmetric part of $Du$ is required to be a measure, and $BD$ is strictly larger than $BV$. For $u \in BD(\GO)$ the jump set, denoted by $J_u,$ is countably $(\mathcal{H}^{n-1}, n-1)$-rectifiable [\ref{bib:Amb.Cos.DM.}], and $Eu$ decomposes as
\begin{equation}
\label{1.2}
Eu:= e(u)\,\mathcal{L}^n + [u]\odot\nu_u \,\mathcal{H}^{n-1}\llcorner J_u + E^{c}u ,
\end{equation}
where $e(u)$ is the density of the absolutely continuous part, $[u]$ is the jump of $u$ across $J_u$ with normal $\nu_u$, the operation $\odot$ is he symmetrized tensor product, and $E^{c}u$ is the Cantor part, see [\ref{bib:Amb.Cos.DM.}] for the details. This decomposition is exactly the one demanded by the mechanics: $e(u)$ is the elastic strain and $J_u$ is the crack. Discarding the diffuse Cantor part gives the subspace $SBD(\GO)$ of special functions of bounded deformation, introduced in analogy with the $SBV$ theory of De Giorgi and Ambrosio [\ref{bib:Amb.DeG.}] and studied in [\ref{bib:Amb.Cos.DM.},\ref{bib:Bel.Cos.DM.}]. One sets
\begin{equation}
\label{1.3}
SBD^p(\GO) := \{u \in SBD(\GO) \ :  \ e(u)\in L^p(\GO),\ \mathcal{H}^{n-1}(J_u)<\infty\}.
\end{equation}
The space $SBD^p$ is still not quite the right one for the Griffith energy. Finiteness of that energy bounds $e(u)$ in $L^p$ and the area of the crack, but it says nothing about the size of the displacement. For the energy minimization problem, it is difficult to obtain a priori $L^\infty$ bounds on  a minimizing sequence. Dal Maso [\ref{bib:DalMaso}] resolved this by defining, through one-dimensional slicing, the space $GBD$ of generalized functions of bounded deformation, roughly speaking, requiring of a measurable $u$ that its directional truncations have bounded directional variation in every direction, and its subspace $GSBD$ in which the slices are in $SBV$. Every 
$u \in GSBD(\GO)$ still possesses an approximate symmetrized gradient $e(u)$ and a countably $(\mathcal{H}^{n-1}, n-1)$-rectifiable jump set $J_u$, both agreeing with the previous notions when $u \in SBD(\GO)$, and one sets
\begin{equation}
\label{1.4}
GSBD^p(\GO) := \{u \in GSBD(\GO): e(u)\in L^p(\GO),\ \mathcal{H}^{n-1}(J_u)<\infty\}.
\end{equation}
These spaces are the appropriate ones for the Griffith energy, accounting for deformation and fracture of a body. Griffith's original criterion [\ref{bib:Griffith}] postulates that a crack advances when the elastic energy released by its growth exceeds the energy needed to create the new surface, the latter being proportional to the created area through a material constant, the toughness. Francfort and Marigo [\ref{bib:Fra.Mar.}] recast this criterion as a global minimization principle in which the displacement and the crack compete in a single functional. For a bounded open set $\GO\subset\mathbb{R}^3$ (body) with toughness $\beta>0$, a displacement $u$ and a closed crack $\Gamma\subset \GO$, the total energy is given by 
\begin{equation}
\label{1.5}
\mathcal{F}(u,\Gamma):= \int_{\GO\setminus\Gamma}W\big(x,e(u)\big)  dx + \beta \mathcal{H}^{2}(\Gamma),
\end{equation}
where $W$ is a convex Carath\'eodory density with $p-$growth in the gradient variable. The bulk term is the stored energy of the uncracked material, here in the linearized setting with $p$-growth ($p=2$ being the physical case), and the surface term penalizes the area of the crack. What distinguishes \eqref{1.5} from a classical elasticity problem is that $\Gamma$ is not prescribed: 
the geometry of the fracture is an unknown on the same footing as $u$, so that nucleation, branching and path selection are outcomes of minimization rather than inputs.

In minimizing \eqref{1.5} over the pairs $(u,\Gamma)$, one follows the $SBV$ theory of De Giorgi and Ambrosio, dropping $\Gamma$ as an independent variable and leting the crack be recorded by the displacement itself through its jump set, so that \eqref{1.5} becomes
the weak Griffith energy
\begin{equation}
\label{1.6}
\mathcal{F}(u): = \int_{\GO}W\big(x,e(u)\big)  dx + \beta\,\mathcal{H}^{2}(J_u), \qquad u \in SBD^p(U).
\end{equation}
The space $GSBD^p$ is the space in which the Griffith energy is coercive: compactness and lower semicontinuity hold there under the natural energy bounds [\ref{bib:Cha.Cri.1}], and the approximate gradient $\nabla u$, which appears on the left-hand side of every Korn inequality below, exists $\mathcal{L}^n$-a.e. for $u \in GSBD^p$ [\ref{bib:Cag.Cha.Sca.}]. The two terms of \eqref{1.6} are structurally mismatched. The bulk term sees only the symmetrized gradient and is blind to the addition of infinitesimal rigid motions, while the surface term carries no information on $\nabla u$ at all. A bound on $\mathcal{G}(u)$ therefore does not by itself automatically control the full gradient. Korn's inequality in $GSBD^p$ is exactly the component which provides that control. Korn's first inequality (\ref{1.1}) fails in $GSBD^p$ in its classical form, for every $p$ and every $n \ge 2$, no matter how small the jump set is. Some forms of preliminary zeroth order Korn-Poincar\'e inequalities were provided by Kohn and Temam and Strang in [\ref{bib:Tem.Str.},\ref{bib:Kohn}]. The appropriate for fracture mechanics formulation of stronger zeroth order Korn-Poincar\'e inequality was provided by Chambolle, Conti, and Francfort [\ref{bib:Cha.Con.Fra.}], where the inequality holds outside an exceptional set controlled by the jump set size in perimeter. For the recent developments of geometric rigidity and Korn's inequality in $GSBD$, we refer the reader to the papers [\ref{bib:Cha.Gia.Pen.},\ref{bib:Cha.Con.Fra.},\ref{bib:Friedrich.1},\ref{bib:Friedrich.2},\ref{bib:Con.Foc.Iur.2},\ref{bib:Cha.Con.Iur.},\ref{bib:Cag.Cha.Sca.}] and references therein. In this manuscript, we will be utilizing the most recent results by Cagnetti, Chambolle, and Scardia in [\ref{bib:Cag.Cha.Sca.}, Theorems 1.1 and 1.2]. Below, we spell out Theorems 1.1 and 1.2 in [\ref{bib:Cag.Cha.Sca.}] explicitly for the convenience of the reader.\\

\noindent \textbf{Theorem A} (Theorem~1.1 of [\ref{bib:Cag.Cha.Sca.}], Cagnetti-Chambolle-Scardia).
\textit{Let $n\in\mathbb{N}$ satisfy $n\ge 2$, let $p\in(1,\infty)$, and let $\GO\subset\mathbb{R}^n$ be a bounded, open and connected Lipschitz set. Then there exists $c=c(n,p,\GO)>0$ with the following property. For any $u\in GSBD^p(\GO)$ there exist a set of finite perimeter
$\omega\subset\GO$ with
\begin{equation}
\label{1.7}
\mathcal{H}^{n-1}(\partial^*\omega)\le c\,\mathcal{H}^{n-1}(J_u),
\end{equation}
and an infinitesimal rigid motion $a$, namely an affine function $a$ with $e(a)=0$, such that
\begin{equation}
\label{1.8}
\int_{\GO\setminus\omega}|\nabla u-\nabla a|^p dx \le c(n,p,\GO)\int_{\GO}|e(u)|^p dx .
\end{equation}
Moreover, there exists $c=c(n,p,q,\GO)>0$ such that
\begin{equation}
\label{1.9}
\|u-a\|_{L^q(\GO\setminus\omega)} \le c(n,p,q,\GO) \|e(u)\|_{L^p(\GO)},
\end{equation}
with $q\le p^*$ if $p<n$, $q<\infty$ if $p=n$, and $q\le\infty$ if
$p>n$.
}

\vspace{0.5cm}

Note, that by the isoperimetric inequality, (\ref{1.7}) also controls the volume of the exceptional set too, namely  
$\big(\mathcal{L}^n(\omega)\big)^{(n-1)/n} \le c \mathcal{H}^{n-1}(J_u)$.\\

\noindent \textbf{Theorem B} (Theorem~1.2 of [\ref{bib:Cag.Cha.Sca.}], Cagnetti--Chambolle--Scardia).
\textit{Let $n\in\mathbb{N}$ with $n\ge 2$, let $p\in(1,\infty)$, and let $\GO\subset\mathbb{R}^n$ be a bounded and open Lipschitz set. Then there exists $c=c(n,p,\GO)>0$ with the following property. For any $u\in GSBD^p(\GO)$ there exist a set of finite perimeter
$\omega\subset\GO$ with $\mathcal{H}^{n-1}(\partial^*\omega)\le c\,\mathcal{H}^{n-1}(J_u)$ and a function 
$v\in W^{1,p}(\GO;\mathbb{R}^n)$ such that
\begin{equation}
\label{1.10}
u=v \ \text{ in } \ \GO\setminus\omega,
\qquad
\int_{\GO}|e(v)|^p dx \le c\int_{\GO}|e(u)|^p dx .
\end{equation}
If in addition $u$ is bounded, then
$\|v\|_{L^\infty(\GO)}\le\|u\|_{L^\infty(\GO)}$.
}

\vspace{0.5cm}

\begin{remark}
\label{Rem:1.1}
It is proven in [\ref{bib:Con.Zwi.}, Theorem~5.10], that open, bounded, connected $(L, R)$-Lipschitz subsets of $\mathbb R^n$ (see Definition~1.2), support Korn's first inequality with a constant depending only on $L,R,p,$ and $n.$ As noted by the authors in [\ref{bib:Cag.Cha.Sca.}], in fact Theorems~A-B hold for open, bounded, Lipschitz subsets of $\mathbb R^n$ with a constant depending only on $N,r,L,$ that are defined in [\ref{bib:Cag.Cha.Sca.}, Remark~4.3], where $r$ is the size of a each chart of a standard uniform cover by cylinders of the boundary $\partial\GO$, and $N$ is the number of the cylinders. On the other hand the constant is invariant under scalings $\GO\to\Gl\GO$ by  [\ref{bib:Cag.Cha.Sca.}, Theorem~4.1], thus it does not depend on $r.$ A straightforward calculation will show that the Lipschitz domain in [\ref{bib:Cag.Cha.Sca.}, Remark~4.3] with parameters $(N,r,L)$ is an $(L_0, R_0)$-Lipschitz domain with $L_0=\max(1,L)$ and $N\leq (1+10R_0\sqrt{1+4L_0^2})^n.$ Hence, combining [\ref{bib:Cag.Cha.Sca.},\ref{bib:Con.Zwi.}], we infer, that the constant $c$ in Theorem A depends only on $L,R,p,$ and $n$ for open, bounded, connected $(L, R)$-Lipschitz subsets of $\mathbb R^n.$ This will be important in the theorems we will prove for thin domains. 
 \end{remark}

Rrecall the uniform Lipschitz framework introduced for Korn's inequality and the Geometric Rigidity Estimate by Conti and Zwicknagl in [\ref{bib:Con.Zwi.}]. The following is Definition 5.1 of [\ref{bib:Con.Zwi.}]. 

\begin{definition}[Uniformly (L,R)-Lipschitz domains]
\label{Def:1.2}
Let $L, R>0$. An open set $\Omega \subseteq \mathbb{R}^n$ is $(L, R)$-Lipschitz if there is $\varepsilon>0$ such that:
\begin{enumerate}
 \item $\mathrm{diam}(\Omega)<R\varepsilon.$ 
 \item For every $x \in \partial \Omega$ there is $f_x \in \operatorname{Lip}\left(\mathbb{R}^{n-1} ; \mathbb{R}\right)$ 
 with $\operatorname{Lip}\left(f_x\right) \leq L,$ and an isometry $\BA_x: \mathbb{R}^n \rightarrow \mathbb{R}^n$ 
 such that $B_\varepsilon(x) \cap \Omega=B_\varepsilon(x) \cap V_x$, where
$$
 V_x:=\BA_x\left\{\left(y^{\prime}, y_n\right) \in \mathbb{R}^{n-1} \times \mathbb{R}: y_n<f_x\left(y^{\prime}\right)\right\}
 $$
\end{enumerate}
\end{definition}

In the present manuscript, we prove the analogue of Theorems~A-B in thin domains (in particular shells) for $GSBD^p$ fields $u,$  
both with and without boundary conditions on the thin face of the shell boundary, under a natural assumption of smallness of the jump sets. The boundary conditions are zero boundary conditions imposed on the thin face of the boundary of the shell. The constant $C$ in (\ref{1.8}) has the same scaling $h^{-p}$ in both cases as in the classical situation mentioned above. Those estimates are proven via a thickness-scale patch-wise Sobolev regularization of $GSBD$ fields, where the constants are scale-free and uniform, thanks to 
Theorems~A-B above. 

The Korn inequality without boundary conditions was proven for square plates in the PhD thesis of my former student Andre Martins Rodrigues reported in his dissertation (Theorem~5.5.1 in [\ref{bib:Rodrigues}]) in October 2023 (published online in March, 2024).
The present paper extends it to $C^{1,1}$ thin domains, for application to fracture of shells. The geometry and combinatorics of the covering of the domain and the chaining of the patches in the inequality without boundary conditions is similar to [\ref{bib:Rodrigues}], which requires extra caution and effort to work, due to the fact that the domain mid-surface is not an exact rectangle and is not flat, and one has to construct a suitable cover of the domain into uniformly $(L,R)$-Lipschitz patches. The inequality with boundary conditions requires some new input too. 

In a parallel development, in the companion paper [\ref{bib:Harutyunyan.3}], we do dimension reduction in shell deformation and fracture analysis  within the above framework of the variational theory of brittle fracture of Francfort and Marigo
[\ref{bib:Fra.Mar.}]. Dimension reduction for brittle thin plates has been done in [\ref{bib:Alm.Tas.1}].


\section{Setting and main results}
\setcounter{equation}{0}
\label{Sec:2}

This section is devoted to the problem setup and formulation of the main results on Korn's first inequalities in thin domains. Assume 
$S \subset \mathbb{R}^3$ is a $C^{1,1}$ compact, connected, embedded regular surface, either without or with boundary, that is a finite disjoint union of embedded $C^{1,1}$ curves, with unit normal field $n(x)\colon S\to \mathbb S^{2}.$ For $h>0$ small, let the family of Lipschitz functions (thickness functions) $g_1^h(x),g_2^h(x)\colon S\to (0,\infty)$ satisfy the uniform conditions
\begin{equation}
\label{2.1}
h\leq g_1^h(x),g_2^h(x)\leq c_1 h\quad \text{and}\quad |\nabla g_1^h(x)|+|\nabla g_2^h(x)|\leq c_2h
\quad\text{for all}\quad x\in S.
\end{equation}
The set 
\begin{equation}
\label{2.2}
\Omega_h=\Big\{x+tn(x) \ : \ x\in S,\ t\in \left(-g_1^h(x),g_2^h(x)\right) \Big\}
\end{equation}
is a thin domain of thickness of order $h$ with mid-surface $S.$ The thin domain $\Omega_h$ is also called a shell with variable thickness. In the particular case when $g_1^h(x)=g_2^h(x)\equiv 1,$ the set $\GO_h$ becomes a shell with thickness $h.$ Let 
$\kappa_0$ dente the maximum of the absolute values of principal curvatures of $S$ and $\tau_S = \mathrm{reach}(S)$
 denote the reach of $S.$ Apparently $\kappa_0$ and $\tau_S$ are positive and finite by $C^{1,1}$ regularity and compactness of 
 $S$ (Section~\ref{Sec:3}). The first main result is the following Korn's first inequality without boundary conditions.

\begin{theorem}
\label{Thm:2.1}
Let $p \in (1,\infty)$ and let $S$ and $\GO_h$ be as in (\ref{2.1})-(\ref{2.2}). There exist constants $h_0, \delta_0, C> 0$, all depending only on $S$ and $p,$ such that for all $h \in (0,h_0)$ and all $u \in GSBD^p(\GO_h)$ with $\mathcal{H}^2(J_u) < \delta_0 h,$ the exist a set of finite perimeter $\omega \subseteq \GO_h$ with $\mathcal{L}^3(\omega) \le C \mathcal{H}^2(J_u),$ a Sobolev field 
$v\in W^{1,p}(\GO_h),$ a constant skew-symmetric matrix $\BA \in \mathbb{R}^{3\times3}_{\mathrm{skew}}$, and a vector 
$b \in \mathbb{R}^3,$ such that
\begin{equation}
\label{2.3}
v=u\quad\text{in}\quad \GO_h\setminus\omega,\qquad  \|e(v)\|_{L^p(\GO_h\setminus\omega)} \le C\|e(u)\|_{L^p(\GO_h)} 
\end{equation}
and 
\begin{equation}
\label{2.4}
 \|\nabla u - \BA\|_{L^p(\GO_h\setminus\omega)} \le \frac{C}{h}\|e(u)\|_{L^p(\GO_h)}
 \end{equation}
 and
\begin{equation}
\label{2.5} 
  \|u(x) - (\BA x + b)\big\|_{L^q(\GO_h\setminus\omega)} \le \frac{C}{h^{1+\frac1p-\frac1q} }\|e(u)\|_{L^p(\GO_h)},
\end{equation}
where $1\leq q\le p^*=\frac{3p}{3-p}$ if $p<3$, $1\leq q<\infty$ if $p=3$, and $1\leq q\le\infty$ if $p>3$. The scaling of the constant in both (\ref{2.4}) and (\ref{2.5}) is optimal as $h\to 0$ for all surfaces $S$ that contain a flat region (open part of a hyperplane in $\mathbb R^3$). 
\end{theorem}

For the Korn inequality with boundary conditions we assume $\partial S \ne \emptyset$. Denote the lateral boundary of $\Omega_h$ as 
\begin{equation}
\label{2.6}
\partial_S\Omega_h=\Big\{x+tn(x) \ : \  x \in \partial S,  \ t\in \left(-g_1^h(x),g_2^h(x)\right)\Big\},
\end{equation}
and define the subspace
\begin{equation}
\label{2.7}
V_0 := \Big\{u \in GSBD^p(\Omega_h) \ : \ u
= 0\ \ \mathcal{H}^2\text{-a.e.\ on}  \ \ \partial_S\Omega_h \Big\},
\end{equation}
i.e., $u$ vanishes on the thin faces of the domain boundary in the sense of traces. Recall that for $u\in GSBD^p(\Omega_h)$ 
and a Lipschitz boundary, the trace is well defined $\mathcal{H}^2$-a.e., see [\ref{bib:DalMaso}].

\begin{theorem}
\label{Thm:2.2}
Assume $\partial S \ne \emptyset$. There exist $h_0, \delta_0, C > 0,$ all depending only on $S$ and $p,$ such that for all  
$h \in (0,h_0)$ and all $u \in V_0$ with $\mathcal{H}^2(J_u) < \delta_0 h,$ there exist a set of finite perimeter 
$\omega \subseteq \Omega_h$ with $\mathcal{L}^3(\omega) \le C\,\mathcal{H}^2(J_u)$ and a Sobolev field $v\in W^{1,p}(\GO_h)$ such that
\begin{equation}
\label{2.8}
v=u\quad\text{in}\quad \GO_h\setminus\omega,\qquad  \|e(v)\|_{L^p(\GO_h\setminus\omega)} \le C\|e(u)\|_{L^p(\GO_h)}
\end{equation}
and 
\begin{equation}
\label{2.9} 
\|\nabla u\|_{L^p(\GO_h\setminus\omega)}+h^{\frac1p-\frac1q}\|u(x)\|_{L^q(\GO_h\setminus\omega)} 
\le \frac{C}{h}\|e(u)\|_{L^p(\GO_h)}, 
\end{equation}
where $1\leq q\le p^*=\frac{3p}{3-p}$ if $p<3$, $1\leq q<\infty$ if $p=3$, and $1\leq q\le\infty$ if $p>3$. 
The scaling of the constants in (\ref{2.7}) are optimal as $h\to 0$ for all surfaces $S$ that contain a flat region.
\end{theorem}

\begin{remark}
\label{Rem:2.3}
A perimeter bound $\mathcal{H}^{n-1}(\partial^*\omega)\le C \mathcal{H}^{n-1}(J_u)$ is impossible for thin domains, and only a 
volume bound $\mathcal{L}^3(\omega) \le C\,\mathcal{H}^2(J_u)$ can be proven as in Theorems~2.1 and 2.2. This is due to the 
fact, that a crack of size $\delta h$ in a shell (for any fixed $\delta \in (0,\delta_0)$), can cut off a piece of shell of diameter of order 
$\delta,$ which has to be included in the exceptional set and which will have perimeter of order $\delta^2,$ thus can not be 
bounded in perimeter by  $C\delta h$ for small $h.$ Hence, the volume bound 
$\mathcal{L}^3(\omega) \le C\,\mathcal{H}^2(J_u)$ is optimal in Theorems~2.1 and 2.2.
\end{remark}


\section{Proof of the main results}
\setcounter{equation}{0}
\label{Sec:3}

\subsection{Proof of Theorem~\ref{Thm:2.1}}

The proof of Theorems~2.1 and 2.2 will make use of covering $S$ with uniform $h-$scale charts and lifting them to the shell. 
The lifted "curvilinear cylindrical" components will satisfy the uniform overlapping property as well as each of them will have a volume comparable to $h^3.$ The "cylinders" containing a small fixed portion of the crack will be called  \textit{good}, and the ones containing a larger portion of the crack will be called \textit{bad}. The strategy will be to apply the uniform Korn's first inequality in Theorems~A-B to the displacement $u$ in every good cylinder, and then connect them all in one skew-symmetric matrix $\BA.$ For this purpose, only a connected cover is not sufficient, and one needs a checkerboard-like chart of $S,$ because the technique requires a few paths for each connection. A similar connected but somewhat disorganized cover of the shell has been constructed in [Lemma~6.1, \ref{bib:Har.Rod.1}], which served the purpose of [\ref{bib:Har.Rod.1}], but will not do here, although it can be slightly modified to serve in the proof of Theorems~2.1 and 2.2 too. Because the geometry of the construction is not something complex, we will only formulate the construction lemma and sketch the proof ideas, which is Lemma~\ref{Lem:3.1} below. A vital property of the covers is that each patch of a chart has to support Korn's first inequality (\ref{1.8}) for $GSBD^p$ vector fields with a uniform constant depending only on $p$ and $S.$ Covers with uniforms $(L,R)-$Lipschitz patches will serve the purpose. We will assume for simplicity, that $S$ is a single patch of the finite atlas covering it. The atlas is finite by compactness. Then putting piecewise Korn's inequalities together on an atlas of connected patches (lifted cylinders) is routine and completely algebraic in nature, thus the procedure can be carried out for $GSBD^p$ fields too and we will skip it here. This is bearing in mind that we can assume, that we have a small enough portion of crack in each of the patches and, and the neighboring overlapping patches intersect at large enough volumes compared to the total crack size. The notation $\Phi(x,t)=x+tn(x)$ for $x\in S$ and $t\in \left(-g_1^h(x),g_2^h(x)\right)$ will be frequently used in the sequel. Also, the for $a,b>0$ the symbolics $a\simeq b$ means that $a$ and $b$ are comparable through uniform constants. 

\begin{lemma}
\label{Lem:3.1}
There exist constants $a_1,a_2,a_3,h_0,N_0,L_0,M_0> 0$ depending only on $S,$ $p$ and $c_1,c_2$ in (\ref{2.1}), such that 
for all $h\in (0,h_0),$ the surface $S$ can be covered by a chart of $N^2$ (or $2N^2$) patches $S_{ij}\subset S \ $ $(1\leq i,j\leq N)$, with $N\simeq 1/h,$ such that the lifted curvilinear cylindrical patches $P_{ij}=\Phi\big(S_{ij}\times\left(-g_1^h(x),g_2^h(x)\right)\big)$
have the properties below.
\begin{itemize}
\item[(P1)] The patches cover the thin domain with the same overlap pattern as the surface grid.
\item[(P2)] For every $1\leq i,j\leq N$ pair one has $a_1h^3 \le \mathcal{L}^3(P_{ij}) \le a_2h^3$ and 
$a_1'h \le \mathrm{diam}(P_{ij}) \le a_2'h.$
\item[(P3)] For face neighbors $|i-i'|+|j-j'|=1$ one has $\mathcal{L}^3(P_{ij}\cap P_{i'j'}) \ge a_3h^3$.
\item[(P4)] Every point of $\Omega_h$ lies in at most $N_0$ patches.
\item[(P5)] Every patch $P_{ij}$ is an open connected $(L_0,M_0)-$Lipschitz domain, thus supports Korn's first inequality 
with a uniform constant depending only on $S,p,c_1,c_2$ by Theorem~B and Remark~\ref{Rem:1.1}.
\item[(P6)] One has $\cup_{i,j=1}^NP_{ij}=\Omega_h$.
\end{itemize}
\end{lemma}

\begin{proof}[Proof of Lemma~\ref{Lem:3.1}]
 
First of all note, that $a_1$ and $a_2$ can be expressed in terms of $a_1'$ and $a_2',$ but we keep the notation for convenience.  Assume without loss of generality that the coordinate patch for $S$ is the box $Q,$ which is the unit cube $Q_0=[0,1]^2$ for internal patches and half of it for boundary patches. Choose the positive integer $N\simeq 1/h$ and divide $Q$ into a grid of $N^2$ equal squares for internal patches and $2N^2$ equal squares for boundary patches, then increase each cube by a factor of $4/3,$ and take the truncated cover of $Q.$ For initially chosen small $\sigma_0=\sigma_0(S)>0$ and for small $h,$ for each small square 
$Q_{ij}$, there exists a $C^2$ diffeomorphism  $\psi: Q_{ij} \to \mathbb{R}^3$ onto its image with $\|D\psi - \BI\| \le \sigma_0$ up to an isometry. This means geometrically, that the image $S_{ij}=\psi(Q_{ij})$ is almost a flat square. Next, the mapping $\Phi(x,t) = x + tn(x)$ is injective on $S\times(-\tau_S,\tau_S).$ It has been shown in the proof of Theorem~2.4 (ii) in [\ref{bib:Har.Rod.1}], that there exists $\tilde h_0=\tilde h_0(S,p)>0$ such that for $h\in (0,\tilde h_0)$, $\Phi$ is bi-Lipschitz from $S\times(-h,h)$ onto $\Omega_h$ with numerical Lipschitz constants. A straightforward calculation gives the Jacobian $J\Phi(x,t) = |(1-t\kappa_1)(1-t\kappa_2)|$, where 
$\kappa_i$ are the principal curvatures of $S.$ Hence, for measurable $E\subset S$ and for $h<1/(2\kappa_0),$ one has
\begin{equation}
\label{3.1}
\frac{1}{4}h \mathcal{H}^2(E) \le \mathcal{L}^3\big(\Phi(E\times(-h,h))\big)
\le \frac{9}{4} h \mathcal{H}^2(E).
\end{equation}
Now, (P1)-(P4) follow from the cover construction, the mentioned properties of $\Phi$ and $\psi,$ and (\ref{3.1}). In this particular case when $S$ is a single patch, one can take $N_0=4,$ and in general case one has to take $N_0=4N_1,$ where $N_1$ is the number of the patches in the atlas of $S.$ The uniform $(L_0,R_0)-$Lipschitz property of $P_{ij}$ appears in the proof of Lemma~6.1 
in [\ref{bib:Har.Rod.1}], needing only slight modifications. Thus $P_{ij}$ supports Korn's inequality with a uniform $C(S,n,p)$ constant 
by Theorems~A-B and Remark~\ref{Rem:1.1} due to the fact that the inequality is scale-invariant. This completes the proof of the lemma.
\end{proof}
We are now ready to prove Theorem~\ref{Thm:2.1}

\begin{proof}[Proof of Theorem~\ref{Thm:2.1}]
Let $h_0$ be as in Lemma~\ref{Lem:3.1} and fix $h< h_0$ and $u \in GSBD^p(\Omega_h)$ with $\mathcal{H}^2(J_u)<\delta h$,  
where $0<\delta\leq\delta_0$ and $\delta_0$ will be defined later in the proof. The constants $C,C_i>0$ in the proof may depend 
only on $S,n,p,q,c_1,c_2$ unless otherwise specified. We will separate the proof of constructing the field $v$ and proving of each of the estimates in the theorem as they do not exactly flow together. The reason for separating the construction of $v$ is that there is no Korn's first inequality on thin domains with "random" holes. Also, the cases $1\leq q<p,$ $p=1$ and $p<q<p^\ast$ are qualitatively different, while for the gradient estimate we have one exponent $p$ only. The proof runs as follows: We first prove Korn's first inequality in (\ref{2.4}), then we prove Poincar\'e-Korn inequality (\ref{2.5}) for the case $1\leq q\leq p.$ For the case $q>p,$ the chaining argument we use in the proof of (\ref{2.4}) does not work directly, yielding an extra negative power of $h,$ and thus a weaker and non-optimal estimate. Then we construct the field $v$ in (\ref{2.3}). All in all, some steps will be made within the same box. The proof proceeds in ten steps.\\

\noindent\textbf{Step 1: Local Korn and the volume bound.}
Due to the scale-invariance properties of the results in Theorems~A-B, we have By (P5), that for each $(i,j)$ there are a set 
$\omega_{ij}\subset P_{ij}$ of finite perimeter, a Sobolev field $v_{ij}\in W^{1,p}(P_{ij}),$ and an infinitesimal rigid motion $a_{ij}=\BA_{ij}x+b_{ij}$ on $P_{ij}$ such that 
\begin{align}
\label{3.2}
& u=v_{ij}\quad\text{in}\quad P_{ij}\setminus\omega_{ij},
 \qquad  \mathcal{H}^2(\partial^*\omega_{ij}) \le C_0 \mathcal{H}^2(J_u\cap P_{ij}), 
\qquad \int_{P_{ij}}|e(v)|^p \le C_0\int_{P_{ij}}|e(u)|^p \\ \nonumber 
& \int_{P_{ij}}|\nabla v_{ij} - \BA_{ij}|^p \le C_0\int_{P_{ij}}|e(u)|^p,\qquad  
\|v_{ij} - a_{ij}\|_{L^q(P_{ij})} \le c_0 h^{1+\frac{3}{q}-\frac{3}{p}}\|e(u)\|_{L^p(P_{ij})}, \\ \nonumber 
& \int_{P_{ij}\setminus\omega_{ij}}|\nabla u - \BA_{ij}|^p \le C_0\int_{P_{ij}}|e(u)|^p,\qquad  
\|u - a_{ij}\|_{L^q(P_{ij}\setminus\omega_{ij})} \le c_0 h^{1+\frac{3}{q}-\frac{3}{p}}\|e(u)\|_{L^p(P_{ij})},
\end{align}
where $C_0=C_0(S,n,p,c_1,c_2),c_0=c_0(S,n,p,q,c_1,c_2)>0$ are the constants in Theorems~A-B, and the scaling $h^{3/q-3/p}$ comes from the scale-invariance as the volume quotient in the Sobolev imbedding inequality on $P_{ij}\setminus\omega_{ij}.$
We have by Lemma~\ref{Lem:3.1} (P2) the bound  $\mathcal{L}^3(\omega_{ij})^{1/3} \le (a_2h^3)^{1/3}$, thus we get by the isoperimetric inequality, that
\begin{equation}
\label{3.3}
\mathcal{L}^3(\omega_{ij})
= \mathcal{L}^3(\omega_{ij})^{2/3}\,\mathcal{L}^3(\omega_{ij})^{1/3}
\le C_1 h\,\mathcal{H}^2(J_u\cap P_{ij}).
\end{equation}
The factor $h$ in (\ref{3.3}) is essential, as it will make the so-called \textit{bad-patch}
count scale as $h^{-2}$ and will make $\mathcal{L}^3(\omega)\le C\mathcal{H}^2(J_u)$ later in Step~4.\\

\noindent\textbf{Step 2: Good and bad patches.}
With the constant $C_1$ in (\ref{3.3}), $a_3$ in (P3) and $\lambda= a_3/4C_1$, we call a patch $P_{ij}$ \textit{good} if 
$$
\mathcal{H}^2(J_u\cap P_{ij}) \le \lambda h^2
$$ 
and \textit{bad} otherwise. Denote the indexing sets of  
\textit{good} and \textit{bad} patches as
\begin{equation}
\label{3.4}
\mathcal{G} := \{(i,j): \mathcal{H}^2(J_u\cap P_{ij}) \le \lambda h^2\}, \qquad \mathcal{B} := \mathcal{G}^c.
\end{equation}
For good patches (\ref{3.3}) gives
\begin{equation}
\label{3.5}
\mathcal{L}^3(\omega_{ij}) \le C_1h\cdot\lambda h^2 = \frac{a_3}4 h^3.
\end{equation}
By (P4) of the lemma we have
$$
\sum_{i,j}\mathcal{H}^2(J_u\cap P_{ij}) \le N_0 \mathcal{H}^2(J_u),
$$ 
and each bad patch contributes at least $\lambda h^2$ crack area, thus we get 
\begin{equation}
\label{3.6}
\#(\mathcal{B}) \le \frac{N_0 \mathcal{H}^2(J_u)}{\lambda h^2},
\end{equation}
and using $\mathcal{H}^2(J_u)<\delta h$ we get 
\begin{equation}
\label{3.7}
\#(\mathcal{B}) \le C_2 \delta N.
\end{equation}
These bounds serve different purposes. The first controls the measure of $\omega$, the second controls the combinatorics. 
This will be important in the combinatorial observations in the later steps.\\

\noindent\textbf{Step 3: Comparing $a_{ij}$ and $v_{ij}$ across a \textit{good} face--overlap.}
Let $(i,j)$, $(i',j')$ be face neighbors, both \textit{good}. For
$$
V= (P_{ij}\cap P_{i'j'})\setminus(\omega_{ij}\cup\omega_{i'j'}),
$$
we have by (P3) and (\ref{3.5}), that
$$
\mathcal{L}^3(V) \ge a_3h^3 - 2\cdot\frac{a_3}4h^3=\frac{a_3}2h^3
$$
Thus, since $\BA_{ij}-\BA_{i'j'}$ is constant, we get
\begin{align}
\label{3.8}
\frac{a_3}2h^3\,|\BA_{ij}-\BA_{i'j'}|^p &\le \int_{V}|\BA_{ij}-\BA_{i'j'}|^p\\ \nonumber
& \le 2^{p-1} \int_{V} \big(|\nabla u - \BA_{ij}|^p + |\nabla u - \BA_{i'j'}|^p\big)\\ \nonumber
& \le C \int_{P_{ij}\cup P_{i'j'}} |e(u)|^p,
\end{align}
also by (\ref{3.2}) we can estimate
\begin{align*}
\|a_{ij}-a_{i'j'}\|_{L^q(V)} & \le \|u-a_{ij}\|_{L^q(V)}+\|u-a_{i'j'}\|_{L^q(V)}\\ \nonumber
 & \le   \|u-a_{ij}\|_{L^q(P_{ij}\setminus\omega_{ij})} + \|u-a_{i'j'}\|_{L^q(P_{i'j'}\setminus\omega_{i'j'})} \\ \nonumber
& \le c_0h^{1+\frac{3}{q}-\frac{3}{p}}\|e(u)\|_{L^p(P_{ij}\cup P_{i'j'})},
\end{align*} 
thus taking into account the fact that $a_{ij}-a_{i'j'}$ is affine, and the geometry of the sets $P_{ij}\cap P_{i'j'}$ and $V$ that also have comparable volumes, we can deduce from the above estimate, the bound 
\begin{equation}
\label{3.9}
\|a_{ij}-a_{i'j'}\|_{L^q(P_{ij}\cap P_{i'j'})} \leq Ch^{1+\frac{3}{q}-\frac{3}{p}}\|e(u)\|_{L^p(P_{ij}\cup P_{i'j'})}.
\end{equation}
Hence for any other patch $P_{kl}$ (volume $\le a_2h^3$ by (P2)), we have by (\ref{3.8}):
\begin{align}
\label{3.10}
\|\BA_{ij}-\BA_{i'j'}\|^p_{L^p(P_{kl})} & \le a_2h^3\,|\BA_{ij}-\BA_{i'j'}|^p \\ \nonumber
& \le C\Big(\|e(u)\|^p_{L^p(P_{ij})} +\|e(u)\|^p_{L^p(P_{i'j'})}\Big).
\end{align}
For the function $w_{ij}'=v_{ij}-v_{i'j'}\in W^{1,p}(P_{ij}\cap P_{i'j'}),$ we have by (\ref{3.9}) and (\ref{3.2}), that
\begin{align}
\label{3.11}
\|w_{ij}'\|_{L^q(P_{ij}\cap P_{i'j'})} 
& \leq \|v_{ij}-a_{ij}\|_{L^q(P_{ij}\cap P_{i'j'})}+\|a_{i'j'}-v_{i'j'}\|_{L^q(P_{ij}\cap P_{i'j'})}+ \|a_{ij}-a_{i'j'}\|_{L^q(P_{ij}\cap P_{i'j'})} 
\\ \nonumber
& \le C h^{1+\frac{3}{q}-\frac{3}{p}}\Big(\|e(u)\|_{L^p(P_{ij})} +\|e(u)\|_{L^p(P_{i'j'})}\Big).
\end{align}

\noindent\textbf{Step 4: Bound on the excised (bad) set.}
Now we call the patch row $P_{i,-}= \bigcup_j P_{ij}$ good, if all its patches are good, and bad  otherwise. Same for the
 columns $P_{-,j}.$  Let $\mathcal{B}^r, \mathcal{B}^c$ be the bad row and column indexing sets.  Obviously we have 
 $\#(\mathcal{B}^r),\#(\mathcal{B}^c) \le \#(\mathcal{B})$. Define
\begin{equation}
\label{3.12}
\omega := \Big(\bigcup_{i\in\mathcal{B}^r,\,j\in\mathcal{B}^c}P_{ij}\Big)
\cup \Big(\bigcup_{ij}\omega_{ij}\Big).
\end{equation}
Every bad patch lies in a bad row and a bad column, hence in $\omega$. By (P2), (\ref{3.3}), (\ref{3.6}) and (\ref{3.7}) 
we can estimate
\begin{align}
\label{3.13}
\mathcal{L}^3(\omega) &\le a_2h^3\cdot\frac{N_0\mathcal{H}^2(J_u)}{\lambda h^2}\cdot
C_2\delta N + C_1h N_0 \mathcal{H}^2(J_u) \\ \nonumber
& \le C(\delta+h) \mathcal{H}^2(J_u).
\end{align}

\noindent\textbf{Step 5: The \textit{good} path system.}
We enumerate good columns $g_1<\dots<g_G$ and bad columns $b_1<\dots<b_B$ where we set 
\begin{equation}
\label{3.14}
G = N-\#(\mathcal{B}^c), \qquad B = \#(\mathcal{B}^c) \le C_2\delta N.
\end{equation}
Fix now $\delta_0$ with $C_2\delta_0 \le \frac13$. Then $B < G$ by (\ref{3.14}) and assign to each bad column $b_l$ its own good column $g_l$ for $l=1,\dots,B$. Fix a good row $i\notin\mathcal{B}^r$ and an anchor $\BA_{i1}$ ($P_{i1}$ is good since its row is). Patches lying in (bad row $\cap$ bad column) sit inside $\omega$, so we can estimate
\begin{equation}
\label{3.15}
\int_{\Omega_h\setminus\omega} |\nabla u -\BA_{i1}|^p
\le \sum_{l=1}^G\sum_{k=1}^{N} \|\nabla u - \BA_{i1}\|^p_{L^p(P_{kg_l}\setminus\omega_{kg_l})}
+ \sum_{l=1}^B\sum_{k\notin\mathcal{B}^r}
\|\nabla u - \BA_{i1}\|^p_{L^p(P_{kb_l}\setminus\omega_{kb_l})}.
\end{equation}
We divide the good patches into two types.\\

\noindent\textbf{Type (a) target patch $P_{kg_l}$ in a good column.} These are the patches that are reachable from $P_{i1}$ through 
the L-shaped path
$$
P_{i1}\to\cdots\to P_{ig_l}\to\cdots\to P_{kg_l},
$$
that runs along good row $i$, then along good column $g_l$. All patches are good, all steps are face-adjacent, and the length 
of the path is at most $2N$.\\

\noindent\textbf{Type (b) target $P_{kb_l}$ in a bad column but good row $k.$} Those are the patches that are reachable 
from $P_{i1}$ through the Z-shaped path
$$
P_{i1}\to\cdots\to P_{ig_l}\to\cdots\to P_{kg_l}\to\cdots\to P_{kb_l},
$$
that runs along good row $i$, good column $g_l$, then good row $k$. The final leg crosses bad columns but stays in the good 
row $k$, so every patch on it is good and the length is at most $3N$. The injective assignment $b_l\mapsto g_l$ keeps the 
accumulation additive. The two path types are illustrated in Figure~1.

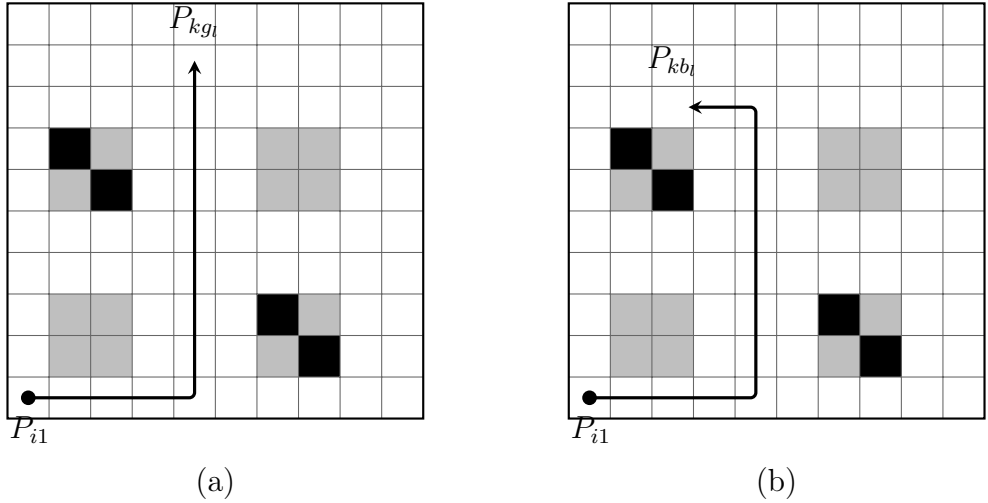
\begin{figure}[!h]
\centering
\begin{tikzpicture}[scale=0.55]
 panel (a)
\begin{scope}
\foreach \r/\c in {6/2, 7/3, 6/7, 6/8, 7/7, 7/8, 2/2, 2/3, 3/2, 3/3,
                   3/8, 2/7}
  \fill[black!25] (\c-1,\r-1) rectangle (\c,\r);
\foreach \r/\c in {7/2, 6/3, 3/7, 2/8}
  \fill[black] (\c-1,\r-1) rectangle (\c,\r);
\draw[step=1, black!60, thin] (0,0) grid (10,10);
\draw[black, thick] (0,0) rectangle (10,10);
\draw[very thick, rounded corners=2pt, -stealth]
  (0.5,0.5) -- (4.5,0.5) -- (4.5,8.62);
\fill (0.5,0.5) circle (5pt and 5pt);
\node[below left, inner sep=1pt] at (1.1,0.12) {$P_{i1}$};
\node[above, inner sep=2pt] at (4.5,9.05) {$P_{kg_l}$};
\node[below] at (5,-0.9) {(a)};
\end{scope}
 panel (b)
\begin{scope}[xshift=13.5cm]
\foreach \r/\c in {6/2, 7/3, 6/7, 6/8, 7/7, 7/8, 2/2, 2/3, 3/2, 3/3,
                   3/8, 2/7}
  \fill[black!25] (\c-1,\r-1) rectangle (\c,\r);
\foreach \r/\c in {7/2, 6/3, 3/7, 2/8}
  \fill[black] (\c-1,\r-1) rectangle (\c,\r);
\draw[step=1, black!60, thin] (0,0) grid (10,10);
\draw[black, thick] (0,0) rectangle (10,10);
\draw[very thick, rounded corners=2pt, -stealth]
  (0.5,0.5) -- (4.5,0.5) -- (4.5,7.5) -- (2.88,7.5);
\fill (0.5,0.5) circle (5pt and 5pt);
\node[below left, inner sep=1pt] at (1.1,0.12) {$P_{i1}$};
\node[above, inner sep=2pt] at (2.5,8.1) {$P_{kb_l}$};
\node[below] at (5,-0.9) {(b)};
\end{scope}
\end{tikzpicture}

\begin{minipage}{.8\textwidth}
\caption{The two path types of Step~5, drawn on the index grid of one chart (patches are shown disjoint for visualization; in the
construction they overlap). Black cells indicate bad patches, and white cells indicate good patches. Gray cells lie in a bad row and
a bad column and they are absorbed into the excised set $\omega$ and no path visit them.
\textbf{(a)}: Target in a good column, L-shaped path along the anchor row $i$ and the good column $g_l$. 
\textbf{(b)}: Target in a bad column but a good row $k$, Z-shaped path entering the target along row $k$. Every cell
visited by either path lies in a good row or a good column and is therefore good.}
\end{minipage}
\label{fig:paths}
\end{figure}

Now, for a path $\pi$ ($|\pi|\le3N$), ending at $P_{kl}$, we have by (\ref{3.2}) and (\ref{3.9}) the bound
\begin{equation}
\label{3.16}
\|\nabla u - \BA_{i1}\|^p_{L^p(P_{kl}\setminus\omega_{kl})} 
\le CN^{p-1}\Big(\|e(u)\|^p_{L^p(P_{i,-})} + \|e(u)\|^p_{L^p(P_{-,g_l})} + \|e(u)\|^p_{L^p(P_{k,-})}\Big),
\end{equation}
with the last term present only for type (b), since then we would have  
$\pi \subset (\text{row } i) \cup (\text{column } g_l) \cup (\text{row } k)$. Finally, we Insert (\ref{3.16}) into (\ref{3.15}). We have at  most $N$ values of each free index, and bounded overlap of rows and columns, thus we have 
$$
\sum_l\|e(u)\|^p_{L^p(P_{-,g_l})} \le N_0 \|e(u)\|^p_{L^p(\Omega_h)}
$$
and likewise for the rows. Combining all these estimates together yields 
\begin{equation}
\label{3.17}
\int_{\Omega_h\setminus\omega}|\nabla u - \BA_{i1}|^p \le CN^{p+1} \|e(u)\|^p_{L^p(P_{i,-})}
+ CN^p \|e(u)\|^p_{L^p(\Omega_h)}.
\end{equation}

\noindent\textbf{Step 6: Averaging over the anchor row}
The estimate (\ref{3.17}) is not quite what we would like to obtain, due to the additional factor $N$ within the first summand on the
 right. With the anchor fixed, all of order $N^2$ targets route their first step through the $O(1)$ patches adjacent to $P_{i1}$, so 
 some patches generate accumulation of order $N^2$, and $N^{p-1}\cdot N^2$ is the true size of the right side of (\ref{3.17}). The remedy is that the anchor is free and we can sum (\ref{3.17}) over the good rows $i\notin\mathcal{B}^r$, of which there are between 
 $N/2$ and $N$ and use the estimate 
$$
\sum_{i\notin\mathcal{B}^r}\|e(u)\|^p_{L^p(P_{i,-})} \le N_0 \|e(u)\|^p_{L^p(\Omega_h)}.
$$
Consequently, we have the bound
$$
\sum_{i\notin\mathcal{B}^r}\int_{\Omega_h\setminus\omega}
|\nabla u - \BA_{i1}|^p \le CN^{p+1}\|e(u)\|^p_{L^p(\Omega_h)}, 
$$
thus if we choose the index $i\notin\mathcal{B}^r$ so that the integral $\int_{\Omega_h\setminus\omega} |\nabla u - \BA_{i1}|^p$
is the minimal among all of $i\notin\mathcal{B}^r$, we get for such $i$ the desired bound:
\begin{equation}
\label{3.18}
\int_{\Omega_h\setminus\omega}\big|\nabla u - \BA \big|^p \le C N^p \|e(u)\|^p_{L^p(\Omega_h)}
= \frac{C}{h^p} \|e(u)\|^p_{L^p(\Omega_h)},
\end{equation}
with $\BA := \BA_{i1}.$ This completes the proof of (\ref{2.4}).\\

Now, we prove the Poincar\'e-Korn inequality (\ref{2.5}) as the case $q=p$ will be necessary for establishing (\ref{2.3}).\\

 \noindent\textbf{Step 7: The case $q=p.$}
 In this case the chaining in the proof of (\ref{2.4}) runs in the exact same lines with affine maps in place of matrices. The
(\ref{3.9}) and (\ref{3.10}) on a good face-overlap will control $|b_{ij}-b_{i'j'}|$ by the same right-hand side like in (\ref{3.9}). 
Telescoping, anchor averaging, and chart gluing are unchanged. \\

 \noindent\textbf{Step 8: The case $q<p.$} Since $\mathcal{L}^3(\GO_h)\le Ch$, H\"older's inequality on $\GO_h\setminus\omega$ together with (\ref{2.3}) give
\begin{align*}
\|u-(\BA x+b)\|_{L^q(\GO_h\setminus\omega)} & \le \mathcal{L}^3(\GO_h)^{\frac1q-\frac1p}\,\|u-(\BA x+b)\|_{L^p(\GO_h\setminus\omega)} \\ 
& \le C h^{\frac{1}{q}-\frac{1}{p}-1} \|e(u)\|_{L^p(\GO_h)}
\end{align*}
by the already proven case $q=p.$ \\

\noindent\textbf{Step 9: The case $p<q\leq p^\ast.$}  In this case the exact same proof paths as in the steps 1-8 does not directly work, thus we need to modify it a bit. This will increase the excised set, but the same volume control will be kept. Let the smaller unexpanded patches $R_{ij} \subset P_{ij}$ be the ones obtained from the original partition of the surface $S$ without the enlargement of each surface patch by a factor of $4/3$ in the proof of Lemma~3.1, which are the lifted curvilinear prismes of the images of the boxes obtained from the partition of the unit square (or half-boxes for boundary patches). The patches $R_{ij}$ will be uniformly $(L_0,R_0)-$Lipschitz exactly as $P_{ij}$ are such. Fix $q\in (p,p^\ast].$ Let $Q$ be the coordinate chart (box) for $S$ and let 
$\widehat Q:=Q\times(0,1)$ be the three-dimensional reference cylinder. We set for brevity 
\begin{equation}
\label{3.19}
 E_{ij}:=\int_{P_{ij}}|e(u)|^p,\qquad
 E:=\int_{\Omega_h}|e(u)|^p,\qquad
 f_{ij}:=(A_{ij},b_{ij})\in\mathbb R^{3\times3}_{\mathrm{skew}}\times\mathbb R^3.
\end{equation}
By (\ref{3.8})-(\ref{3.9}), it is straightforward, that good-face neighbors satisfy
 \begin{equation}
 \label{3.20} 
 |f_{ij}-f_{i'j'}|^p\le Ch^{-3}(E_{ij}+E_{i'j'}).
 \end{equation}
On $\widehat Q$, choose a variable-separated partition of unity $\eta_{ij}(y)=\eta_i(x_1)\eta_j(x_2)$, independent of $x_3$, with 
$|\nabla\eta_{ij}|\le Ch^{-1}$, such that $\eta_{ij}=1$ on pairwise disjoint central prisms $\widehat K_{ij}$ of volume comparable to 
$h^2$. Here, the plateaus of $\eta_{ij}$ are disjoint, but their supports overlap. Let
$$
 D:=\bigcup_{(i,j)\in\mathcal B}\supp\eta_{ij},\qquad
 f:=\begin{cases}
 \displaystyle\sum_{(i,j)\in\mathcal G}\eta_{ij}f_{ij},&\text{in }\widehat Q\setminus D,\\
 0,&\text{in }D.
 \end{cases}
$$
Then $f\in SBV^p(\widehat Q,\mathbb R^6)$, and $f=f_{ij}$ on every good $\widehat K_{ij}$, where we identify the entry space with 
$\mathbb R^6$. Its jumps lie on $\partial D\cap\widehat Q$, whose area is at most $Ch(\#\mathcal B)$.
Thus (\ref{3.6}), (\ref{3.20}), and the bounded overlap give
\begin{equation}
 \label{3.21}
  \mathcal{H}^2(J_f)\le Ch(\#\mathcal B)\le C\frac{\mathcal{H}^2(J_u)}h,
\end{equation}
and
\begin{equation}
 \label{3.22}
 \int_{\widehat Q}|\nabla f|^p \le Ch^{2-p}\sum_{\substack{(i,j)\sim(i',j')\\(i,j),(i',j')\in\mathcal G}}
 |f_{ij}-f_{i'j'}|^p \le Ch^{-p-1}E.
\end{equation}
where the sum $(i',j')\sim (i,j)$ is taken over the neighboring (overlapping) good patches. Because $\mathcal{H}^2(J_f)\le C\delta$ is small, we apply Theorem~A in the fixed cylinder $\widehat Q$ to $(f_k,0,0)$ for each of the six scalar components $f_k$ of $f$. Since the last two rows of its gradient vanish, the gradient estimate controls every entry of the approximating skew matrix by $C\|\nabla f_k\|_{L^p(\widehat Q)}$. The retained volume is bounded below when $\delta_0$ is small. Consequently the displacement estimate controls $f_k$ up to a constant. Taking the union of the six exceptional sets and using the three-dimensional isoperimetric inequality, gives a set $F\subset\widehat Q$ and
$f_0=(\widetilde \BA,\widetilde b)$ such that by (\ref{3.22}) we will get
\begin{equation}
 \label{3.23}
  \mathcal{L}^3(F)\le C\left(\frac{\mathcal{H}^2(J_u)}h\right)^{\frac{3}{2}},
 \qquad
 \|f-f_0\|_{L^q(\widehat Q\setminus F)}\le Ch^{-1-\frac{1}{p}}E^{\frac{1}{q}}.
\end{equation} 
Now, we call a patch $R_{ij}$ \textit{very good,} if it has not been removed in (\ref{3.12}). We further enlarge the exceptional set 
$\omega$ from (\ref{3.12}) by adding to it the very good patches $R_{ij},$ for which $|F\cap\widehat K_{ij}|\ge |\widehat K_{ij}|/2$, 
and denote the result by $\widetilde \omega$. There are at most $C|F|h^{-2}$ such patches. Hence, using (\ref{3.13}) and 
$\mathcal{H}^2(J_u)<\delta h$, we still get the required volume control 
\begin{align}
 \label{3.24}
  |\widetilde\omega| & \le |\omega|+Ch|F| \\ \nonumber
& \le C(h+\sqrt\delta)\mathcal{H}^2(J_u) \\ \nonumber
 & \le C\mathcal{H}^2(J_u).
\end{align}
 Let $\mathcal I$ consist of the good indices for which $|\widehat K_{ij}\setminus F|>|\widehat K_{ij}|/2$. For finite $q$, we have the obvious bound
$$
 h^2\sum_{(i,j)\in\mathcal I}|f_{ij}-f_0|^q \le C\int_{\widehat Q\setminus F}|f-f_0|^q.
$$
Writing $\widetilde a(x)=\widetilde \BA x+\widetilde b$, and using $|R_{ij}|\le Ch^3,$ $|x|\le C$, and (\ref{3.23}) we obtain 
\begin{align}
 \label{3.25}
\left(\sum_{(i,j)\in\mathcal I} \|a_{ij}-\widetilde a\|_{L^q(R_{ij})}^q\right)^{\frac{1}{q}}
 & \le Ch^{\frac{1}{q}}\|f-f_0\|_{L^q(\widehat Q\setminus F)} \\ \nonumber
 & \le Ch^{-1+\frac{1}{q}-\frac{1}{p}}E^{\frac{1}{p}}.
\end{align}
On the other hand, (\ref{3.2}), bounded overlap, and $q\ge p$ yield
\begin{align}
 \label{3.26} 
 \left(\sum_{(i,j)\in\mathcal I}
 \|u-a_{ij}\|_{L^q(R_{ij}\setminus\omega_{ij})}^q\right)^{\frac{1}{q}} 
 & \le Ch^{1+\frac{3}{q}-\frac{3}{p}} \left(\sum_{i,j}E_{ij}^{\frac{q}{p}}\right)^{\frac{1}{q}} \\ \nonumber
  &   \le Ch^{1+\frac{3}{q}-\frac{3}{p}}E^{\frac{1}{p}},
\end{align}
and every point outside $\widetilde\omega$ belongs to a retained chart. Since $2-2/p+2/q>0,$ we discover
\begin{equation}
 \label{3.27}
  \|u-\widetilde a\|_{L^q(\Omega_h\setminus\widetilde\omega)} \le Ch^{-1+\frac{1}{q}-\frac{1}{p}}E^{\frac{1}{p}}.
\end{equation}
For $q=\infty$, replace the sums by suprema and the same argument works. Finally, we need to replace the rigid motion 
$\widetilde a$ in (\ref{3.27}) by $a(x)=Ax+b$ in (\ref{2.4})-(\ref{2.5}) for $p=q.$ H\"older's inequality applied in (\ref{3.27}) gives
$$
\|u-\widetilde a\|_{L^p(\Omega_h\setminus\widetilde\omega)} \le Ch^{-1}E^{\frac{1}{p}}.
$$
The uniform norm equivalence for rigid motions on the fixed surface patch together with the known $q=p$ estimate, implies
\begin{align*}
 |\widetilde \BA-\BA|+|\widetilde b-b| &\le Ch^{-\frac{1}{p}}\|\widetilde a-a\|_{L^p(\Omega_h\setminus\widetilde\omega)}\\
 &\le Ch^{-1-\frac{1}{p}}E^{\frac{1}{p}}.
\end{align*}
Finally, relabeling $\widetilde\omega$ as $\omega$ completes the proof.\\

Now we extend the field $v$ in (\ref{2.3}) in the so-called very good set.\\

\noindent\textbf{Step 10: Construction of the field $v$ in the very good set.}
Recall, that a pair $(i,j)$ (or equivalently the patch $P_{ij}$) is \textit{very good}, if the patch $P_{ij}$ has not been excised entirely in (\ref{3.12}). This labeling implies that each very good patch is automatically good and
\begin{equation}
\label{3.28}
\GO_h\setminus \Go\subseteq \bigcup_{(i,j)-very \ good} P_{ij}=P.
\end{equation}
We choose smooth cutoff functions $\varphi_{ij} \in C^\infty(P_{ij},\mathbb R)$ relative to $\GO_h,$ such that $0\leq \varphi_{ij}\leq 1$, $\varphi_{ij}=1$ in $R_{ij},$ $\mathrm{supp}(\varphi_{ij})\subseteq P_{ij}$ relative to $\GO_h,$ and $|\nabla \varphi_{ij}(x)|\leq C/h$ uniformly in $(i,j)$ and $h.$ These functions 
will generate a partition of unity on 
\begin{equation}
\label{3.29}
R=\bigcup_{(i,j)-very \ good} R_{ij}
\end{equation}
as follows:
\begin{equation}
\label{3.30}
\psi_{ij}(x)=\frac{\varphi_{ij}(x)}{\varphi(x)}, \qquad \text{where}  
\qquad  \varphi(x)=\sum_{(i,j)-very \ good}\varphi_{ij}(x)
\quad \text{for all}\quad  x\in R,
\end{equation}
and we clearly have by the definition of the excited set $\Go,$ the inclusion 
\begin{equation}
\label{3.31}
P\setminus R\subseteq \Go.
\end{equation}
We have $\psi_{ij}(x)v_{ij}(x)\in W^{1,p}(P_{ij}\cap R,\mathbb R^n),$ thus extending the fields $\psi_{ij}(x)v_{ij}(x)$ as zero into 
$R\setminus P_{ij},$ we obtain $\psi_{ij}(x)v_{ij}(x)\in W^{1,p}(R,\mathbb R^n),$ and define
\begin{equation}
\label{3.32}
v(x)=\sum_{(i,j)-very \ good} \psi_{ij}(x)v_{ij}(x),\qquad x\in R.
\end{equation}
Due to the finite overlap we have that $v\in W^{1,p}(R,\mathbb R^n)$ and $\mathrm{supp}(v)\subseteq P$ relative to $\GO_h.$ 
Thanks to the fact that 
$$
\sum_{(i,j)-very \ good}\nabla \psi_{ij}=0\quad\text{ in}\quad R,
$$ 
we can write 
$$
e(v)=\sum_{(i,j)-very \ good}\psi_{ij}(x)e(v_{ij})+ \sum_{(i,j)-very \ good}\nabla \psi_{ij}(x) \odot (v_{ij}-v),\quad x\in R,
$$
where for vectors $a,b\in\mathbb R^n,$ the operation $\odot $ is the symmetric tensor product given by 
$$
a\odot b=\frac{a\otimes b+b\otimes a}{2}.
$$
Thus we can estimate the first summand by the bounded overlap of $P_{ij}$ and (\ref{3.2}) as 
\begin{equation}
\label{3.33}
\left\|\sum_{(i,j)-very \ good}\psi_{ij}(x)e(v_{ij})\right\|_{L^p(R)}\leq C\|e(u)\|_{L^p(\GO_h)}.
\end{equation}
In the estimation of the second summand, note that 
$$v_{ij}-v=\sum_{(i',j')\sim (i,j)}\psi_{i'j'}(v_{ij}-v_{i'j'})=\sum_{(i',j')\sim (i,j)}\psi_{i'j'}w_{ij}',$$ 
where the sum $(i',j')\sim (i,j)$ is taken over the neighboring (overlapping) patches. Thus, owing to (\ref{3.11}), we can estimate the 
norm $\|w_{ij}'\|_{L^p(R)}$ for diagonal overlap by triangle inequality through a face-overlapping patch and hence, we obtain for each pair $(i,j)$ due to the bounded overlap property of $P_{ij}$ the bound
\begin{align*}
\|v_{ij}-v\|_{L^p(R)} &\leq  \sum_{(i',j')\sim (i,j)}\|w_{ij}'\|_{L^p(R)} \\
& \leq \sum_{(i,'j')\sim (i,j)}\|w_{ij}'\|_{L^p(P_{ij}\cap P_{i'j'})} \\
& \leq C h\Big(\|e(u)\|_{L^p(P_{ij})} +\|e(u)\|_{L^p(P_{i'j'})}\Big).
\end{align*}
Now, the last estimate together with the bound $|\nabla \psi_{ij}(x)|\leq C/h$ in $R,$ and again with the bounded overlap, yields 
\begin{equation}
\label{3.34}
\left\|\sum_{(i,j)-very \ good}\nabla \psi_{ij}(x) \odot (v_{ij}-v)\right\|_{L^p(R)}\leq C\|e(u)\|_{L^p(\GO_h)}.
\end{equation}
Finally, putting together (\ref{3.33}) and (\ref{3.34}) we arrive at the bound 
$$\|e(v)\|_{L^p(R)}\leq C\|e(u)\|_{L^p(\GO_h)}.$$
We can now redefine the new excised set as $\Go\cup (p\setminus R),$ which may only double in volume thanks to (\ref{3.22}). 
Consequently, $v$ defined in the very good set $R$ with the new excised set $\Go$ will do the job in $R.$  \\


 \noindent\textbf{Step 11: Sharpness of the scaling $h^{-p}.$}
It is a well-know fact, that for plates and Sobolev fields, a Kirchhoff bending Ansatz $u = (-x_3\partial_1w, -x_3\partial_2w, w)$ 
compactly supported on a small ball contained in the flat region, gives $h^{-p}$ in Korn's first inequality for the classical case and thus for the case of $GSBD^p$ fields too. It is easy to see that same happens in all of the inequalities (\ref{2.4}), (\ref{2.5}) and (\ref{2.9}). This completes the proof of Theorem~\ref{Thm:2.1}.

\end{proof}


\subsection{Proof of Theorem~\ref{Thm:2.2}}
\label{Sub:3.2}

\begin{proof}[Proof of Theorem~\ref{Thm:2.2}.]

Throughout this section it will be assumed that $\partial S \ne \emptyset$. The strategy will be to enlarge the surface, and hence the thin domain through the thin face, then extend the filed $u$ as zero on the extended collar, given the initial zero Dirichlet boundary conditions. We will then apply Theorem~\ref{Thm:2.1} to the extended field on the extended domain, and using the fact that the field is zero on the extended collar, we will be able to bound the norm of the matrix $\BA$ by the elastic energy. We divided the proof into three steps.\\

\noindent\textbf{Step 1: The collar extension.}
Since $\partial S$ is a finite disjoint union of embedded $C^{1,1}$ curves and $S$ is a compact $C^{1,1}$ surface, there are a compact connected $C^{1,1}$ surface $\widetilde S \supset S$ with the same properties as $S,$ and a constant $r_1>0$, both depending 
only on $S$, such that
\begin{equation}
\label{3.35}
\mathcal{C} := \widetilde S\setminus S
\quad\text{contains the collar}\quad
\{x \in \widetilde S:\ 0 < \operatorname{dist}_{\widetilde S}(x,\partial \widetilde S) < r_1\},
\end{equation}
and such that the geometric quantities of $\widetilde S$ are controlled by those of $S$ with numerical constants. Indeed, in the boundary charts of Section~\ref{Sec:3} the surface is the graph of a $C^{1,1}$ function over a half-rectangle with a straight side mapping to $\partial S$. Extending each such graph function to the full rectangle by a $C^{1,1}$ Whitney extension by [\ref{bib:Fefferman}, Problem 1]) with the corresponding norm bounds, and glueing the finitely many extensions by a partition of unity subordinate to the boundary charts, produces $\widetilde S$ with (\ref{3.17}) with comparable to $S$ geometric properties. In the same way we extend the thickness functions $g_i^h$, without relabeling them, in the Lipschitz space according to [\ref{bib:Fefferman}, Problem 1]. The unit normal of $\widetilde S$ restricts to $n$ on $S$, and we keep the notation $n$ for $\widetilde S$ too. Denote the extended thin domain and the collar:
$$
\widetilde\Omega_h := \left\{x + t\,n(x)\ : \ x \in \widetilde S,\ t\in (-g_1^h(x), g_2^h(x))\right\},
\qquad
\mathcal{C}_h := \widetilde\Omega_h\setminus\overline{\Omega}_h ,
$$
so that $\mathcal{C}_h$ is the thickened collar. An application of (\ref{3.1}) on 
$\widetilde S$ gives
\begin{equation}
\label{3.36}
\mathcal{L}^3(\mathcal{C}_h)  \ge  \tfrac12\,h\,\mathcal{H}^2(\mathcal{C}) =c_0h ,
\qquad
\mathcal{L}^3(\widetilde\Omega_h) \le C_0 h ,
\end{equation}
with $C_0,c_0>0$ depending only on $S$. By construction, the surface $\widetilde S$ supports Theorem~\ref{Thm:2.1}, 
with constants depending only on $S$ and the other unchanged parameters. Fix $\widetilde h_0, \widetilde\delta_0, \widetilde C$ 
as in Theorem~\ref{Thm:2.1} for $\widetilde S$ accordingly.

Let now $u \in V_0$ with $\mathcal{H}^2(J_u) < \widetilde \delta_0 h$, and define
\begin{equation}
\label{3.37}
\tilde u := \begin{cases} u & \  \text{in } \ \Omega_h,\\[2pt] 
0 &  \ \text{in } \ \mathcal{C}_h. \end{cases}
\end{equation}
It is straightforward to show, that the extension satisfies the following properties:
\begin{equation}
\label{3.38}
\tilde u \in GSBD^p(\widetilde\Omega_h), \qquad e(\tilde u) = e(u)\chi_{\Omega_h}\quad \mathcal{L}^3\text{-a.e.},
\qquad \mathcal{H}^2\big(J_{\tilde u}\setminus J_u\big) = 0,
\end{equation}
and
\begin{equation}
\label{3.39}
\mathcal{H}^2(J_{\tilde u}) \le \mathcal{H}^2(J_u) < \widetilde \delta_0 h \qquad 
\|e(\tilde u)\|_{L^p(\widetilde\Omega_h)} = \|e(u)\|_{L^p(\Omega_h)}.
\end{equation}
Indeed, the interface between the two regions in (\ref{3.37}) is exactly the lateral face $\partial_S\Omega_h$. Both $\Omega_h$ and 
$\mathcal{C}_h$ are Lipschitz with constants depending only on $S$. This is proven in [\ref{bib:Har.Rod.1}, Lemma 4.2]. For fields of bounded deformation, because $\widetilde u=0$ in  $\mathcal{C}_h$, glueing across a Lipschitz interface gives
$$
E\tilde u \;=\; Eu\,\llcorner\,\Omega_h \;+\; \big(\operatorname{tr}u \odot \nu_S\big)\,
\mathcal{H}^2\llcorner\partial_S\Omega_h ,
$$
where $\nu_S$ is the outer normal of $\Omega_h$ along the lateral face and $\operatorname{tr}u$ the inner trace; the outer trace vanishes because $\tilde u \equiv 0$ on $\mathcal{C}_h$, e.g., [\ref{bib:Babadjian}, Proposition~3.5 ]. Since $u \in V_0$ we have $\operatorname{tr}u = 0$ $\mathcal{H}^2$-a.e. on $\partial_S\Omega_h$, so the interface term vanishes identically. Hence no jump is created along $\partial_S\Omega_h$, the diffuse part of $E\tilde u$ is that of $Eu$ extended by zero, and the Cantor part remains zero. This gives (\ref{3.38})-(\ref{3.39}).\\

\noindent\textbf{Step 2: Estimate on the collar.}
Assume now $\delta \le \widetilde\delta_0$ and $h < \widetilde h_0$, so that Theorem~\ref{Thm:2.1} applies 
to $\tilde u$ on $\widetilde\Omega_h$. Let $\widetilde\omega \subset \widetilde\Omega_h$, $\BA$ and $b$ be the resulting objects,
so that
\begin{equation}
\label{3.40}
\mathcal{L}^3(\widetilde\omega) \le C \mathcal{H}^2(J_{\tilde u}) \le C \mathcal{H}^2(J_u),
\end{equation}
and by (\ref{2.3}) and (\ref{3.38}),
\begin{equation}
\label{3.41}
\|\nabla \tilde u - \BA\|_{L^p(\widetilde\Omega_h\setminus\widetilde\omega)} + 
h^{\frac{1}{p}-\frac{1}{q}}\|\tilde u - \BA x - b\|_{L^q(\widetilde\Omega_h\setminus\widetilde\omega)}
 \le \frac{C}{h}\|e(u)\|_{L^p(\Omega_h)}.
\end{equation}
Set $\omega := \widetilde\omega\cap\Omega_h$. Then we have by \eqref{3.40}, that  
$$
\mathcal{L}^3(\omega) \leq \mathcal{L}^3(\widetilde\omega)\leq C\mathcal{H}^2(J_u),
$$
which is the measure bound claimed in Theorem~\ref{Thm:2.2}. We now use the estimate (\ref{3.41}) in the collar to bound the norm 
of $\BA x+ b.$ Using $\mathcal{H}^2(J_u) < \widetilde\delta_0h$ in (\ref{3.39}) together with (\ref{3.36}), and choosing 
$\widetilde\delta_0$ small enough that $C\widetilde \delta_0 \le \frac{1}{2}c_0$ we can estimate
\begin{equation}
\label{3.42}
\mathcal{L}^3(\mathcal{C}_h\setminus\widetilde\omega) \ge c_0h - C\widetilde\delta_0 h  \ge \tfrac12 c_0 h.
\end{equation}
On $\mathcal{C}_h$ one has $\nabla\tilde u = 0$ and $\tilde u = 0$ pointwise, so restricting  (\ref{3.41}) to 
$\mathcal{C}_h\setminus\widetilde\omega$ and using (\ref{3.42}), we obtain
\begin{align*}
\frac{1}{2}c_0 h |\BA|^p & \le \int_{\mathcal{C}_h\setminus\widetilde\omega}|\nabla\tilde u - \BA|^p\\
& \le \frac{C}{h^p}\|e(u)\|^p_{L^p(\Omega_h)} ,
\end{align*}
which reduces to
\begin{equation}
\label{3.43}
|\BA| \le \frac{C}{h^{\frac{p+1}{p}}} \|e(u)\|_{L^p(\GO_h)} .
\end{equation}
In the same way, since $|\BA x| \le C|\BA|$ on the bounded set $\widetilde\Omega_h$ we have
\begin{align*}
\frac{1}{2} c_0 h |b|^q & \le C\int_{\mathcal{C}_h\setminus\widetilde\omega}
\big(|\tilde u - \BA x - b|^q + |\BA x|^q\big) \\ \nonumber
& \le \frac{C}{h^{1-q-\frac{q}{p}}}\|e(u)\|^q + C h |\BA|^q ,
\end{align*}
hence, by (\ref{3.43}) we discover,
\begin{equation}
\label{3.44}
|b|  \le \frac{C}{h^{\frac{p+1}{p}}}\|e(u)\|_{L^p(\GO_h)} .
\end{equation}
Finally, combining (\ref{3.41}),  (\ref{3.43}), and (\ref{3.44}) we get (\ref{2.9}) by the triangle inequality. 
This completes the proof of Theorem~\ref{Thm:2.2}. 

\end{proof}


\section*{Acknowledgements}
This material is supported by the National Science Foundation under Grants No. DMS-2206239.

\section*{AI disclosure} We used Claude Fable 5.0 and ChatGPT-6 Astra for reference search and editing to improve the presentation of the manuscript. The main ideas and formulations (also in the proofs) were generated and known to us already in 2023, see [\ref{bib:Rodrigues}].


\end{document}